\documentclass[twoside,leqno,twocolumn]{article}
\usepackage{ltexpprt}

\usepackage{amssymb}
\usepackage{%
amsmath,amssymb,amsfonts,
latexsym, amscd, epsfig, epic,color}
\usepackage{mathrsfs}
\usepackage{eucal}
\usepackage{colordvi}

\usepackage{cite}

\usepackage{xcolor}
\usepackage{hyperref}
\definecolor{darkgreen}{rgb}{0,0.4,0}
\definecolor{BrickRed}{rgb}{0.65,0.08,0}
\hypersetup{colorlinks=true,linkcolor=blue,citecolor=red,filecolor=BrickRed,urlcolor=darkgreen}

\newcommand{\polya}{P{\'o}lya~}
\newcommand{\PR}{\mathbb{P}}
\newcommand{\LandauO}{\CMcal{O}}

\newcommand{\qed}{\hfill $\square$}

\newcommand{\E}{\mathbb{E}}

\newcommand{\Tc}{\CMcal{T}}

\newcommand{\Aut}{\operatorname{Aut}}

\newtheorem{example}{Example}
\numberwithin{example}{section}
\newcommand*\samethanks[1][\value{footnote}]{\footnotemark[#1]}

\begin{document}

\title{\Large A note on the scaling limits of random P\'olya trees
}
\author{Bernhard Gittenberger\thanks{Institut f\"{u}r Diskrete Mathematik und Geometrie, Technische Universit\"{a}t Wien, Wiedner Hauptstr. 8--10/104, 1040 Vienna, Austria. Corresponding author: Michael Wallner. Emails: gittenberger@dmg.tuwien.ac.at; yu.jin@tuwien.ac.at; michael.wallner@tuwien.ac.at.
}\\
\and
Emma Yu Jin\samethanks\\
\and
Michael Wallner\samethanks}
\date{}

\maketitle


\begin{abstract} \small
Panagiotou and Stufler (arXiv:1502.07180v2) recently proved one important fact on their way to establish the scaling limits of random \polya trees:
a uniform random \polya tree of size $n$ consists of a conditioned critical Galton-Watson tree $C_n$ and many small forests, where with probability tending to one as $n$ tends to infinity, any forest $F_n(v)$, that is attached to a node $v$ in $C_n$, is maximally of size $\vert F_n(v)\vert=O(\log n)$.
Their proof used the framework of a Boltzmann sampler and deviation inequalities.

In this paper, first, we employ a unified framework in analytic combinatorics to prove this fact with additional improvements on the bound of $\vert F_n(v)\vert$, namely $\vert F_n(v)\vert=\Theta(\log n)$.
Second, we give a combinatorial interpretation of the rational weights of these forests and the defining substitution process
in terms of automorphisms associated to a given \polya tree.
Finally, we derive the limit probability that for a random node $v$ the attached forest $F_n(v)$ is of a given size.
\end{abstract}

\section{Introduction and main results}\label{sec:intro}
First, we recall the asymptotic estimation of the number of \polya trees with $n$ nodes from the literature \cite{poly37,niwi78,otte48}.
Second, we present Theorem~\ref{T:1} that leads to the proof of the scaling limits of random \polya trees in \cite{P:142}.
\subsection{\polya trees}
 A {\em \polya tree} is a rooted unlabeled tree considered up to symmetry.
 The {\em size} of a tree is given by the number of its nodes. We denote by $t_{n}$ the number of \polya trees of size $n$ and by $T(z) = \sum_{n \geq 1} t_n z^n$ the corresponding ordinary generating function.  By P\'{o}lya's enumeration theory \cite{poly37}, the generating function $T(z)$ satisfies
\begin{align}\label{E:penum1}
T(z)=z\exp\left(\sum_{i=1}^{\infty}\frac{T(z^i)}{i}\right).
\end{align}
The first few terms of $T(z)$ are then
\begin{align}\label{eq:polyaimplicit}
\begin{aligned}
T(z)&=z + z^2 + 2 z^3 + 4 z^4 + 9 z^5 + 20 z^6 + 48 z^7 \\
    & \quad + 115 z^8 + 286 z^9 + 719 z^{10} + \cdots,
\end{aligned}
\end{align}
(see OEIS~A$000081$, \cite{Sloane}). By differentiating both sides of (\ref{E:penum1}) with respect to $z$,  one can derive a recurrence relation of $t_n$ (see \cite[Chapter~29]{niwi78} and \cite{otte48}), which is
\begin{align*}
t_n=\frac{1}{n-1}\sum_{i=1}^{n-1}t_{n-i}\sum_{m\vert i}mt_m, \quad \mbox{ for }\, n>1,
\end{align*}
and $t_1=1$.
\polya \cite{poly37} showed that the radius of convergence $\rho$ of $T(z)$ satisfies $0 < \rho < 1$ and that $\rho$ is the unique singularity on the circle of convergence $|z|=\rho$. Subsequently, Otter \cite{otte48} proved that $T(\rho)=1$ as well as the singular expansion
\begin{align}
\label{eq:polyaasympt}
\begin{aligned}
	T(z) &= 1 - b\left(\rho-z\right)^{1/2} + c(\rho-z) \\
	     & \quad + \CMcal{O}\left((\rho-z)^{3/2}\right),
\end{aligned}
\end{align}
where $\rho \approx 0.3383219$, $b \approx 2.68112$ and $c = b^2/3 \approx 2.39614$.

By transfer theorems \cite{FS} he derived
\begin{align*}
	t_n &= \frac{b \sqrt{\rho}}{2 \sqrt{\pi}} \frac{\rho^{-n}}{\sqrt{n^{3}}} \left(1 + \CMcal{O}\left(\frac{1}{n}\right)\right).
\end{align*}

We will see that $T(z)$ is connected with the \emph{exponential generating function} of Cayley
trees. ``With a minor abuse of notation'' (\emph{cf.} \cite[Ex.~10.2]{Janson12_SGT}), Cayley trees belong to the class of \emph{simply generated trees}. Simply generated trees have been introduced by Meir and Moon \cite{memo:78} to describe a weighted version of rooted trees. They are defined by the functional equation
\begin{align*}
	y(z) &= z \Phi(y(z)),\qquad \mbox{ with } \\
	\Phi(z) &=\sum_{j\ge 0}\phi_j\,z^j, \quad \phi_j \geq 0.
\end{align*}
The power series $y(x)=\sum_{n\ge 1}y_nx^n$ has non-negative coefficients and is the generating
function of {\em weighted} simply generated trees. One usually assumes that $\phi_0>0$ and
$\phi_j>0$ for some $j\ge 2$ to exclude the trivial cases. In particular, in the above-mentioned
sense, {\em Cayley trees} can be seen as simply generated trees which are characterized by $\Phi(z) = \exp(z)$.
It is well known that the number of rooted Cayley trees of size $n$ is given by $n^{n-1}$.

Let
\begin{align*}
	C(z) &= \sum_{n \geq 0} n^{n-1} \frac{z^n}{n!},
\end{align*}
be the associated exponential generating function.

Then, by construction it satisfies the functional equation
\begin{align*}
	C(z)=ze^{C(z)}.
\end{align*}
In contrast, \polya trees are not simply generated (see \cite{DG10} for a simple proof of this fact).
Note that though $T(z)$ and $C(z)$ are closely related, \polya trees are not related to Cayley
trees in a strict sense, but to a certain class of weighted unlabeled trees which will be called
$C$-trees in the sequel and have the ordinary generating function $C(z)$. This is precisely the
simply generated tree associated with $\Phi(z) = \exp(z)$, now in the strict sense of the definition
of simply generated trees.

In order to analyze the dominant singularity of $T(z)$, we follow \cite{otte48, poly37}, see also \cite[Chapter~VII.5]{FS}, and we rewrite
\eqref{eq:polyaimplicit} into
\begin{align}
	\label{eq:polyadecoA}
	T(z) &= z e^{T(z)} D(z),\quad\,\mbox{ where }\\
	D(z) &= \sum_{n \geq 0 } d_n z^n = \exp\left(\sum_{i=2}^{\infty}\frac{T(z^i)}{i}\right). \notag
\end{align}
We observe that $D(z)$ is analytic for $|z| < \sqrt{\rho}<1$ and that $\sqrt{\rho}>\rho$.
From~\eqref{eq:polyadecoA} it follows that $T(z)$ can be expressed in terms of the generating function of Cayley trees: Indeed, assume that $T(z)$ is a function $H(zD(z))$ depending on $zD(z)$.
By~\eqref{eq:polyadecoA} this is equivalent to $H(x) = x \exp(H(x))$. Yet, this is the functional
equation for the generating function of Cayley trees. As this functional equation has a unique power series solution we have $H(x) = C(x)$, and we just proved
\begin{align}
	\label{eq:polyadeco}
	T(z) &= C(z D(z) ).
\end{align}
Note that $T(z)=C(zD(z))$ is a case of a super-critical composition schema which is characterized by the fact that the dominant singularity of $T(z)$ is strictly smaller than that of~$D(z)$. In other words, the dominant singularity $\rho$ of $T(z)$ is determined by the outer function~$C(z)$. Indeed, $\rho\,D(\rho)=e^{-1}$, because $e^{-1}$ is the unique dominant singularity of~$C(z)$.

Let us introduce two new classes of weighted combinatorial structures: $D$-forests and $C$-trees.
We set $d_n=[z^n]D(z)$ which is the \emph{accumulated weight} of all \emph{$D$-forests} of size $n$.
These are weighted forests of \polya trees which are constrained to contain for every \polya tree at least two identical copies or none. In other words, if a tree appears in a $D$-forest it has to appear at least twice.
From (\ref{eq:polyaimplicit}) and (\ref{eq:polyadecoA}) one gets its first values
\begin{align}\label{E:D}
	D(z)&= \sum_{n=0}^{\infty}d_n z^n \\
	    &= 1 + \frac{1}{2}z^2 + \frac{1}{3}z^3 + \frac{7}{8}z^4 + \frac{11}{30}z^5 \notag\\
	    & \quad + \frac{281}{144}z^6 +\frac{449}{840}z^7 + \cdots. \notag
\end{align}
The weights are defined in such a way that composition scheme~\eqref{eq:polyadeco} is satisfied.
In Theorem~\ref{T:2} we will make these weights explicit.
From \eqref{eq:polyadecoA} we can derive a recursion of $d_n$. We get
\begin{align*}
d_n=\frac{1}{n}\sum_{i=2}^{n}d_{n-i}\sum_{\substack{m\vert i\\m\ne i}}mt_m, \quad\mbox{ for }\,\,n\ge 2,
\end{align*}
as well as $d_0=1$, and $d_1=0$.

The second concept is the one of $C$-trees, which are weighted \polya trees. The weight is defined by the composition~\eqref{eq:polyadeco}. Let $c_n = [z^n] C(z) = \frac{n^{n-1}}{n!}$ be the \emph{accumulated weight} of all \emph{$C$-trees} of size $n$. In other words, we interpret the exponential generating function of Cayley trees~$C(z)$ as an ordinary generating function of {\em weighted} objects:
\begin{align*}
	C(z) &= \sum_{n \geq 0}  \frac{n^{n-1}}{n!} z^n.
\end{align*}

Informally speaking, the composition~\eqref{eq:polyadeco} can be interpreted as such that a \polya tree is constructed from a $C$-tree where a $D$-forest is attached to each node.

This construction is in general not bijective, because the $D$-forests consist of \polya trees and are not distinguishable from the underlying \polya tree, see Figure~\ref{F:0}.
In general there are different decompositions of a given \polya tree into a $C$-tree and $D$-forests.
Theorem~\ref{T:2} will give a probabilistic interpretation derived from the automorphism group of a \polya tree (see also Example~\ref{ex:probinterpretation}).

\begin{figure}[htbp]
\begin{center}
\includegraphics[scale=1.2]{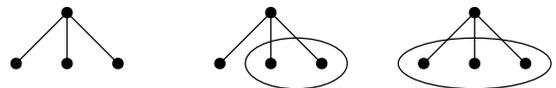}
\caption{The decomposition of a \polya tree with $4$ nodes into a $C$-tree (non-circled nodes) and $D$-forests (circled nodes). For this \polya tree there are $3$ different decompositions. 
\label{F:0}}
\end{center}
\end{figure}

\subsection{Main results}

Consider a random \polya tree of size $n$, denoted by $T_n$, which is a tree that is selected uniformly at random from all \polya trees with $n$ vertices. We use $C_n$ to denote the random $C$-tree that is contained in a random \polya tree $T_n$. For every vertex
$v$ of $C_n$, we use $F_n(v)$ to denote the $D$-forest that is attached to the vertex $v$ in $T_n$, see Figure~\ref{F:1}.
\begin{figure*}[htbp]
\begin{center}
\includegraphics[scale=0.8]{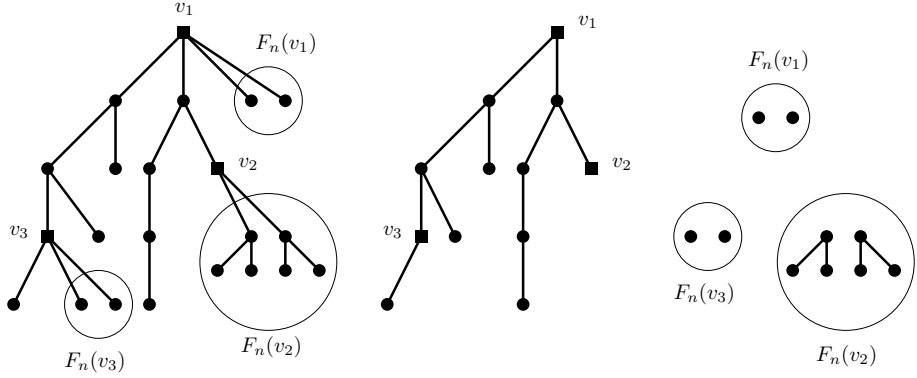}
\caption{A random \polya tree $T_n$ (left), a (possible) $C$-tree $C_n$ (middle) that is contained in $T_n$ where all $D$-forests $F_n(v)$, except $F_n(v_1),F_n(v_2),F_n(v_3)$ (right), are empty.\label{F:1}}
\end{center}
\end{figure*}

Let $L_n$ be the maximal size of a $D$-forest contained in $T_n$, that is, $\vert F_n(v)\vert \le L_n$ holds for all $v$ of $C_n$ and the inequality is sharp. For the upper bound see also \cite[Eq.~(5.5)]{P:142}.
\begin{theorem}\label{T:1}
For $0<s<1$,
\begin{align}\label{E:main1}
\begin{aligned}
&(1-(\log n)^{-s})\left(\frac{-2\log n}{\log\rho}\right)\le L_n \le \\
& (1+(\log n)^{-s})\left(\frac{-2\log n}{\log\rho}\right)
\end{aligned}
\end{align}
holds with probability $1-o(1)$.
\end{theorem}
Our first main result is a new proof of Theorem~\ref{T:1} by applying the unified framework of Gourdon \cite{Gourdon}. Our second main result is a combinatorial interpretation of all weights on the $D$-forests and $C$-trees in terms of automorphisms associated to a given \polya tree.

Let $c_{n,k}$ denote the cumulative weight of all $C$-trees of size $k$ that are contained in
\polya trees of size $n$. By $t_{c,n}(u)$ and $T_c(z,u)$ we denote the corresponding generating
function and the bivariate generating function of $\left(c_{n,k}\right)_{n,k\ge 0}$, respectively, that is,
\begin{align*}
	t_{c,n}(u) &= \sum_{k=1}^n c_{n,k}u^k \quad \mbox{ and } \\
	T_c(z,u) &= \sum_{n\ge 0}t_{c,n}(u)z^n.
\end{align*}
Note that $c_{n,k}$ is in general not an integer. By marking the nodes of all $C$-trees in \polya trees we find a functional equation for the bivariate generating function ${T}_c(z,u)$, which is
\begin{align}\label{E:bi1}
\begin{aligned}
{T}_c(z,u) &= zu\exp\left({T}_c(z,u)\right)\exp\left(\sum_{i=2}^{\infty}\frac{{T}(z^i)}{i}\right)\\
	         &= zu\exp\left({T}_c(z,u)\right)D(z).
\end{aligned}
\end{align}
For a given permutation $\sigma$ let $\sigma_1$ be the number of fixed points of $\sigma$.
Our second main result is the following:
\begin{theorem}\label{T:2}
Let $\CMcal{T}$ be the set of all \polya trees, and $\operatorname{MSET}^{(\geq 2)}(\CMcal{T})$ be the multiset (or forest) of \polya trees where each tree appears at least twice if it appears at all.
Then the cumulative weight $d_n$ (defined in \eqref{E:D}) of all such forests of size $n$ satisfies
\begin{align*}
	d_n = \sum_{\substack{F \in \operatorname{MSET}^{(\geq 2)}( \CMcal{T} )\\|F| = n}} \frac{|\{ \sigma \in \Aut(F) ~|~ \sigma_1 = 0\}|}{|\Aut(F)|}
\end{align*}
where $\Aut(F)$ is the automorphism group of $F$ (see Definition~\ref{D:aut} section~\ref{S:size-D}). Furthermore, the polynomial associated to $C$-trees in \polya trees of size $n$ is given by
\begin{align*}
	t_{c,n}(u)& =\sum_{
	T \in \CMcal{T}, \; \vert T\vert=n} t_{T}(u), \quad \mbox{ where } \\
	t_{T}(u) &= \frac{1}{|\Aut(T)|} \sum_{\sigma \in \Aut(T)} u^{\sigma_1}.
\end{align*}
In particular, for all $T\in \CMcal{T}$, it holds that $t_T'(1)=|\CMcal{P}(T)|$
where $\CMcal{P}(T)$ is the set of all trees which are obtained by pointing (or coloring) one single node in $T$.
\end{theorem}
For a given \polya tree $T$ the polynomial $t_{T}(u)$ gives rise to a probabilistic interpretation of the composition scheme~\eqref{eq:polyadeco}. For a given tree $T$, the weight of $u^k$ in the polynomial $t_T(u)$ is the probability that the underlying $C$-tree is of size $k$. In other words, $t_{T}(u)$ is the probability generating function of the random variable $C_T$ of the number of $C$-tree nodes in the tree $T$ defined by
\begin{align}
	\label{eq:probCT}
	\PR(C_T = k) := [u^k] t_{T}(u).
\end{align}
This random variable $C_T$ is a refinement of $T_n$ in the sense that
\begin{align*}
	\PR(C_T = k) = \PR\left(|C_n| = k ~|~ T_n = T\right).
\end{align*}

Finally, we derive the limiting probability that for a random node $v$ the attached forest $F_n(v)$ is of a given size. This result is consistent with the Boltzmann sampler from \cite{P:142}. The precise statement of our third main result is the following:
\begin{theorem}\label{T:312}
	The generating function $T^{[m]}(z,u)$ of \polya trees, where each vertex is marked by $z$, and each weighted $D$-forest of size $m$ is marked by $u$, is given by
	\begin{align}\label{E:bi2}
		T^{[m]}(z,u) &= C \left( u z d_m z^m + z \left( D(z) - d_m z^m \right) \right),
	\end{align}
where $d_m=[z^m]D(z)$.
The probability that the $D$-forest $F_n(v)$ attached to a random $C$-tree node $v$ is of size $m$ is given by
	\begin{align*}
		\PR\left( \vert F_n(v)\vert = m \right) = \frac{d_m \rho^m}{D(\rho)} \left( 1 + \LandauO\left(n^{-1}\right)\right).
	\end{align*}
\end{theorem}
\subsection{Paper outline}
The paper is organized as follows. In Section~\ref{S:D-forest} we prove Theorem~\ref{T:1} and discuss the size of the $C$-tree $C_n$ in a random \polya tree $T_n$. In Section~\ref{S:size-D} we prove Theorems~\ref{T:2} and \ref{T:312}. In Section~\ref{S:last} we conclude with final remarks.

\section{The maximal size of a \texorpdfstring{$D$}{D}-forest}
\label{S:D-forest}
We will use the generating function approach from \cite{Gourdon} to analyze the maximal size $L_n$ of $D$-forests in a random \polya tree $T_n$, which provides a new proof of Theorem~\ref{T:1}. Following the same approach, we can establish a central limit theorem for the random variable $\vert C_n\vert$, which has been done in \cite{Ben2} for the more general random $\CMcal{R}$-enriched trees.

{\em Proof of Theorem~\ref{T:1}}.
In (5.5) of \cite{P:142}, only an upper bound of $L_n$ is given. By directly applying Gourdon's results (Theorem~4 and Corollary $3$ of \cite{Gourdon}) for the super-critical composition schema, we find that for any positive $m$,
\begin{align*}
\mathbb{P}[L_n\le m]&=\exp\left(-\frac{c_1n}{m^{3/2}}\rho^{m/2}\right)(1+\CMcal{O}(\exp(-m\varepsilon))), \\
c_1 &\sim \frac{b}{2\sqrt{\pi}(1-\sqrt{\rho})(D(\rho)+\rho D'(\rho))},
\end{align*}
as $n\to \infty$.
Moreover, the maximal size $L_n$ satisfies asymptotically, as $n\rightarrow \infty$,
\begin{align*}
\mathbb{E}L_n &= -\frac{2\log\,n}{\log \rho}-\frac{3}{2}\frac{2}{\log \rho}\log\log\,n+\CMcal{O}(1)\quad\mbox{ and } \\
\mathbb{V}\mbox{ar}\,L_n &= \CMcal{O}(1).
\end{align*}
By using Chebyshev's inequality, one can prove that $L_n$ is highly concentrated around the mean~$\mathbb{E}L_n$. We set $\varepsilon_n=(\log n)^{-s}$ where $0<s<1$ and we get
\begin{eqnarray*}
\mathbb{P}(\vert L_n-\mathbb{E}L_n\vert\ge \varepsilon_n\cdot\mathbb{E}L_n)
\le\frac{\mathbb{V}\mbox{ar}\,L_n}{\varepsilon_n^2\cdot(\mathbb{E}L_n)^2}=o(1),
\end{eqnarray*}
which means that Relation (\ref{E:main1}) holds with probability $1-o(1)$.
\qed

It was shown in \cite{Ben2} that the size $\vert C_n\vert$ of the $C$-tree $C_n$ in $T_n$ satisfies a central limit theorem and $\vert C_n\vert=\Theta(n)$ holds with probability $1-o(1)$. In particular see \cite[Eq.~(3.9) and~(3.10)]{Ben2}, and \cite[Eq.~(5.6)]{P:142}. The precise statement is the following.
\begin{theorem}\label{theo:ctreesize0}
The size of the $C$-tree $\vert C_n\vert$ in a random \polya tree $T_n$ of size $n$ satisfies a central limit theorem where the expected value $\mathbb{E}\vert C_n\vert$ and the variance $\mathbb{V}\mbox{ar}\,\vert C_n\vert$ are asymptotically
\begin{align*}
\mathbb{E}\vert C_n\vert &= \frac{2n}{b^2\rho}(1 + \CMcal{O}(n^{-1})),\quad\mbox{ and }\\
\mathbb{V}\mbox{ar}\,\vert C_n\vert &= \frac{11n}{12b^2\rho}
(1+\CMcal{O}(n^{-1})).
\end{align*}
Furthermore, for any $s$ such that $0<s<1/2$, with probability $1-o(1)$ it holds that
\begin{align}\label{E:concen}
(1-n^{-s})\frac{2n}{b^2\rho}\le \vert C_n\vert \le (1+n^{-s})\frac{2n}{b^2\rho}.
\end{align}
\end{theorem}
Random \polya trees belong to the class of random $\CMcal{R}$-enriched trees and we refer the readers to \cite{Ben2} for the proof of Theorem~\ref{theo:ctreesize0} in the general setting. Here we provide a proof of Theorem~\ref{theo:ctreesize0} to show the connection between a bivariate generating function and the normal distribution and to emphasize the simplifications for the concrete values of the expected value and variance in this case.

{\em Proof of Theorem~\ref{theo:ctreesize0} (see also \cite{Ben2})}. It follows from \cite[Th.~2.23]{Drmotabook} that the random variable $\vert C_n\vert$ satisfies a central limit theorem. In the present case, we set
$F(z,y,u)=zu\exp(y)D(z)$. It is easy to verify that $F(z,y,u)$ is an analytic function when $z$
and $y$ are near $0$ and that $F(0,y,u)\equiv 0$, $F(x,0,u)\not\equiv 0$ and all coefficients
$[z^ny^m]F(z,y,1)$ are real and non-negative. From \cite[Th.~2.23]{Drmotabook} we know that
$T_c(z,u)$ is the unique solution of the functional identity $y=F(z,y,u)$. Since all coefficients
of $F_y(z,y,1)$ are non-negative and the coefficients of $T(z)$ are positive as well as
monotonically increasing, this implies that $(\rho,T(\rho),1)$ is the unique solution of $F_y(z,y,1)=1$, which leads to the fact that $T(\rho)=1$. Moreover, the expected value is
\begin{align*}
\mathbb{E}\vert C_n\vert &= \frac{nF_u(z,y,u)}{\rho F_z(z,y,u)} \\
&= \frac{[z^n]\partial_u{T}_c(z,u)\vert_{u=1}}{[z^n]{T}(z)} \\
&= \left([z^n]\frac{{ T}(z)}{1-{T}(z)}\right)\left([z^n]{T}(z)\right)^{-1} \\
&= \frac{2n}{b^2\rho}(1 + \CMcal{O}(\frac1n)).
\end{align*}
The asymptotics are directly derived from~\eqref{eq:polyaasympt}. Likewise, we can compute the variance
\begin{align*}
\mathbb{V}\mbox{ar}\,\vert C_n\vert
&=\frac{[z^n]{T}(z)(1-{T}(z))^{-3}}{[z^n]{T}(z)}
-(\mathbb{E}\,\vert C_n\vert)^2 \\
&=\frac{11n}{12b^2\rho}
(1+\CMcal{O}(n^{-1})).
\end{align*}
Furthermore, $\vert C_n\vert$ is highly concentrated around $\mathbb{E}\,\vert C_n\vert$, which can be proved again by using Chebyshev's inequality. We set $\varepsilon_n=n^{-s}$ where $0<s<1/2$ and get
\begin{align*}
\mathbb{P}(\big\vert \vert C_n\vert-\mathbb{E}\vert C_n\vert\big\vert\ge \varepsilon_n\cdot\mathbb{E}\vert C_n\vert)
\le&\frac{\mathbb{V}\mbox{ar}\vert C_n\vert}{\varepsilon_n^2\cdot(\mathbb{E}\vert C_n\vert)^2} \\
=& \CMcal{O}(n^{2s-1})=o(1),
\end{align*}
which yields (\ref{E:concen}).
\qed

As a simple corollary, we also get the total size of all weighted $D$-forests in $T_n$. Let $\CMcal{D}_n$ denote the union of all $D$-forests in a random \polya tree $T_n$ of size $n$.
\begin{corollary}\label{C:1}
	The size of weighted $D$-forests in a random \polya tree of size $n$ satisfies a central limit theorem where the expected value $\mathbb{E}\vert\CMcal{D}_n\vert$ and the variance $\mathbb{V}\mbox{ar}\vert\CMcal{D}_n\vert$ are asymptotically
	\begin{align*}
	\mathbb{E}\vert\CMcal{D}_n\vert&=n\left(1-\frac{2}{b^2\rho}\right)(1 + \CMcal{O}(n^{-1})),\quad\mbox{ and } \\
	\mathbb{V}\mbox{ar}\vert\CMcal{D}_n\vert&=\frac{11n}{12b^2\rho}
	(1+\CMcal{O}(n^{-1})).
	\end{align*}
\end{corollary}
Theorem~\ref{theo:ctreesize0} and Corollary~\ref{C:1} tell us that a random \polya tree $T_n$ consists mostly of a $C$-tree (proportion $\frac{2}{b^2\rho}$ comprising $\approx 82.2\%$ of the nodes) and to a small part of $D$-forests (proportion $1-\frac{2}{b^2\rho}$ comprising $ \approx 17.8\%$ of the nodes). Furthermore, the average size of a $D$-forest $F_n(v)$ attached to a random $C$-tree vertex in $T_n$ is $\frac{b^2\rho}{2}-1 \approx 0.216$, which indicates that on average the $D$-forest $F_n(v)$ is very small, although the maximal size of all $D$-forests in a random \polya tree $T_n$ reaches $\Theta(\log n)$.
\begin{Remark}
Let us describe the connection of (\ref{eq:polyadeco}) to the Boltzmann sampler from \cite{P:142}. We know that
$F(z,y,1)=z\Phi(y)D(z)$ where $\Phi(x)=\exp(x)$ and $y=T(z)$. By dividing both sides of this equation by $y=T(z)$, one obtains from~\eqref{eq:polyadecoA} that
\begin{align*}
1=\frac{z{D}(z)}{{T}(z)}\exp({T}(z))
=\exp(-{T}(z))\sum_{k\ge 0}\frac{{T}^k(z)}{k!},
\end{align*}
which implies that in the Boltzmann sampler $\Gamma T(x)$, the number of offspring contained in the $C$-tree $C_n$ is Poisson distributed with parameter $T(x)$. As an immediate result, this random $C$-tree $C_n$ contained in the Boltzmann sampler $\Gamma T(\rho)$ is a critical Galton-Watson tree since the expected number of offspring is $F_y(z,y,1)=1$ which holds only when $(z,y)=(\rho,1)$.
\end{Remark}
\section{\texorpdfstring{$D$-forests and $C$-trees}{D-forests and C-trees}}
\label{S:size-D}


In order to get a better understanding of $D$-forests and $C$-trees, we need to return to the original proof of \polya on the number of \polya trees \cite{poly37}. The important step is the treatment of tree automorphisms by the cycle index. Let us recall what it means that two graphs are isomorphic.

\begin{Definition}\label{D:aut}
	Two graphs $G_1$ and $G_2$ are \emph{isomorphic} if there exists a bijection between the vertex sets of $G_1$ and $G_2$,
	$f : V(G_1) \to V(G_2)$ such that two vertices $v$ and $w$ of $G_1$ are adjacent if and only if $f(v)$ and $f(w)$ are adjacent in $G_2$. If $G_1=G_2$ we call the bijection $f$ an \emph{automorphism}. The automorphism group of the graph $G_1$ is denoted by $\Aut(G_1)$.
\end{Definition}

For any permutation $\sigma$, let $\sigma_i$ be the number of cycles of length $i$ of $\sigma$. We define the {\em type} of $\sigma$, to be the sequence $(\sigma_1,\sigma_2,\ldots,\sigma_k)$ if $\sigma\in S_k$. Note that $k=\sum_{i=1}^k i\sigma_i$.
\begin{Definition}[Cycle index]\label{D:cin}
Let $G$ be a subgroup of the symmetric group $S_k$. Then, the {\em cycle index} is
\begin{align*}
Z(G; s_1,s_2,\ldots,s_k)=\frac{1}{|G|}\sum_{\sigma\in G}s_1^{\sigma_1}s_2^{\sigma_2}\cdots s_k^{\sigma_k}.
\end{align*}
\end{Definition}
Now we are ready to prove Theorem~\ref{T:2}.

\subsection{Proof of Theorem~\ref{T:2}} By P\'{o}lya's enumeration theory \cite{poly37}, 
the generating function $T(z)$ satisfies the functional equation
\begin{align*}\nonumber
T(z)&=z \sum_{k \ge 0} Z(S_k; T(z),T(z^2),\ldots,T(z^k))\\
&=z\sum_{k\ge 0}\frac{1}{k!}\sum_{\sigma\in S_k}(T(z))^{\sigma_1}(T(z^2))^{\sigma_2}\cdots (T(z^k))^{\sigma_k},
\end{align*}
which can be simplified to \eqref{E:penum1}, the starting point of our research, by a simple
calculation. However, this shows that the generating function of $D$-forests from~\eqref{eq:polyadecoA} is given by
\begin{align*}
	D(z) &= 
	        \exp\left(\sum_{i=2}^\infty \frac{T(z^i)}{i}\right)\\
			&= \sum_{k \geq 0} Z(S_k;0,T(z^2),\ldots,T(z^k))\\
			&= \sum_{k \geq 0}\frac{1}{k!}\sum_{
			\sigma\in S_k,\; \sigma_1=0}(T(z^2))^{\sigma_2}\cdots (T(z^k))^{\sigma_k}.
\end{align*}
This representation enables us to interpret the weights $d_n$ of $D$-forests of size $n$:
A $D$-forest of size $n$ is a multiset of $k$ \polya trees, where every tree occurs at least twice.
Its weight is given by the ratio of fixed point free automorphisms over the total number of automorphisms.
Equivalently, it is given by the number of fixed point free permutations $\sigma \in S_k$ of these trees rescaled by the total number of orderings $k!$.

Let $\Tc$ be the set of all \polya trees and $\operatorname{MSET}^{(\geq 2)}( \Tc )$ be the multiset of \polya trees where each tree appears at least twice if it appears at all.
Combinatorially, this is a forest without unique trees.
Then, their weights are given by
\begin{align*}
	d_n = \sum_{\substack{F \in \operatorname{MSET}^{(\geq 2)}( \Tc )\\|F| = n}} \frac{|\{ \sigma \in \Aut(F) ~|~ \sigma_1 = 0\}|}{|\Aut(F)|}.
\end{align*}

\begin{example}
	The smallest $D$-forest is of size $2$, and it consists of a pair of single nodes. There
is just one fixed point free automorphism on this forest, thus $d_2 = 1/2$. For $n=3$ the forest
consists of $3$ single nodes. The fixed point free permutations are the $3$-cycles, thus $d_3 = 2/6 = 1/3$. The case $n=4$ is more interesting. A forest consists either of $4$ single nodes, or of $2$ identical trees, each consisting of $2$ nodes and one edge. In the first case we have $6$ $4$-cycles  and $3$ pairs of transpositions. In the second case we have $1$ transposition swapping the two trees. Thus, $d_4 = \frac{6+3}{24} + \frac{1}{2} = \frac{7}{8}$.
\end{example}


These results also yield a natural interpretation of $C$-trees. We recall that by definition
\begin{align*}
	T_c(z,u) &= \sum_{n \geq 0} t_{c,n}(u) z^n,
\end{align*}
where $t_{c,n}(u)=\sum_{k}c_{n,k}u^k$ is the polynomial marking the $C$-trees in \polya trees of size $n$. From the decompositions~\eqref{eq:polyadeco} and (\ref{E:bi1}) we get the first few terms:
\begin{align*}
	t_{c,1}(u) &= u, \\
	t_{c,2}(u) &= u^2, \\
	t_{c,3}(u) &= \frac{3}{2} u^3 + \frac{1}{2} u, \\
	t_{c,4}(u) &= \frac{8}{3} u^4 + u^2 + \frac{1}{3} u %
	.
\end{align*}
Evaluating these polynomials at $u=1$ obviously returns $t_{c,n}(1) = t_n$, which is the number of \polya trees of size $n$.
Their coefficients, however, are weighted sums depending on the number of $C$-tree nodes. For a given \polya tree there are in general several ways to decide what is a $C$-tree node and what is a $D$-forest node. The possible choices are encoded in the automorphisms of the tree, and these are responsible for the above weights as well.

Let $T$ be a \polya tree, and $\Aut(T)$ be its automorphism group. For an automorphism $\sigma \in \Aut(T)$ the nodes which are fixed points of $\sigma$ are $C$-tree nodes. All other nodes are part of $D$-forests. Summing over all automorphisms and normalizing by the total number gives the $C$-tree generating polynomial for $T$:
\begin{align}
  \label{E:tTu}
  \begin{aligned}
	t_{T}(u) &= Z(\Aut(T); u, 1,\ldots,1) \\
	         &= \frac{1}{|\Aut(T)|} \sum_{\sigma \in \Aut(T)} u^{\sigma_1}.
  \end{aligned}
\end{align}
The polynomial of $C$-trees in \polya trees of size $n$ is then given by
\begin{align*}
	t_{c,n}(u) &= \sum_{
	T \in \CMcal{T},\; \vert T\vert=n} t_{T}(u).
\end{align*}

\begin{example}
	\label{ex:probinterpretation}
	For $n=3$ we have $2$ \polya trees, namely the chain $T_1$ and the cherry $T_2$. Thus, $\Aut(T_1) = \{ \text{id} \}$, and $\Aut(T_2) = \{\text{id}, \sigma\}$, where $\sigma$ swaps the two leaves but the root is unchanged. Thus,
	\begin{align*}
		t_{T_1}(u) &= u^3, \\
		t_{T_2}(u) & = \frac{1}{2} (u^3 + u).
	\end{align*}
	For $n=4$ we have $4$ \polya trees shown in Figure~\ref{F:2}. Their automorphism groups are given by
\begin{align*}
	\Aut(T_1) &= \{ \text{id} \},\\
	\Aut(T_2) &= \{ \text{id} \},\\
	\Aut(T_3) &= \{ \text{id}, (v_3\, v_4) \} \cong S_2,\\
	\Aut(T_4) &= \{\text{id}, (v_2\,v_3), (v_3\,v_4), (v_2\,v_4), \\
	          & \qquad (v_2\, v_3\, v_4), (v_2\,v_4\,v_3) \} \cong S_3.
\end{align*}
This gives
	\begin{align*}
		t_{T_1}(u) &= u^4, \\
		t_{T_2}(u) &= u^4, \\
		t_{T_3}(u) &= \frac{1}{2} (u^4+u^2), \\
		t_{T_4}(u) &= \frac{1}{6} (u^4 + 3u^2 + 2u).
	\end{align*}
This enables us to give a probabilistic interpretation of the composition scheme~\eqref{eq:polyadeco}. For a given tree the weight of $u^k$ is the probability that the underlying $C$-tree is of size $k$. In particular, $T_1$ and $T_2$ do not have $D$-forests. The tree $T_3$ consists of a $C$-tree with $4$ or with $2$ nodes, each case with probability $1/2$. In the second case, as there is only one possibility for the $D$-forest, it consists of the pair of single nodes which are the leaves. Finally, the tree $T_4$ has either $4$ $C$-tree nodes with probability $1/6$, $2$ with probability $1/2$, or only one with probability $1/3$. These decompositions are shown in Figure~\ref{F:0}.
\end{example}
\begin{figure}[htbp]
\begin{center}
\includegraphics[scale=1.1]{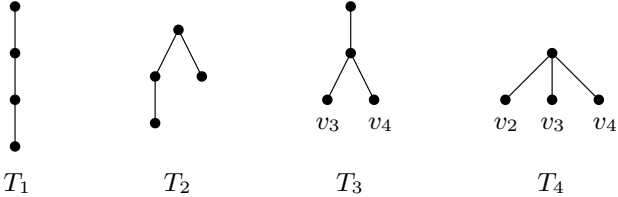}
\caption{All \polya trees of size $4$\label{F:2}.}
\end{center}
\end{figure}

In the same way as we got the composition scheme in \eqref{eq:polyadeco}, we can rewrite $T_c(z,u)$ from~\eqref{E:bi1} into $T_c(z,u)=C(uzD(z))$. 
The expected total weight of all $C$-trees contained in all \polya trees of size $n$ is the $n$-th coefficient of $T_c(z)$, which is
\begin{align}
	\label{E:expTc}
	\begin{aligned}
	T_c(z) &:= \left.\frac{\partial}{\partial u} T_c(z,u) \right|_{u=1} \\
	       &= \frac{T(z)}{1 - T(z)} \\
	       &= z + 2z^2 + 5z^3 + 13z^4 + 35 z^5 + \cdots.
	\end{aligned}
\end{align}
Let us explain why these numbers are integers, although the coefficients of $t_{c,n}(u)$ are in general not. We will show an even stronger result. Let $T$ be a tree and $\CMcal{P}(T)$ be the set of all trees with one single pointed (or colored) node which can be generated from $T$.

\pagebreak

\begin{lemma}
	\label{L:markingTc}
	For all $T \in \CMcal{T}$ it holds that $t_T'(1) = |\CMcal{P}(T)|$.
\end{lemma}
\begin{proof}
From~\eqref{E:tTu} we get that
\begin{align*}
	t_T'(1) = \sum_{\sigma \in \Aut(T)} \frac{\sigma_1}{\vert\Aut (T)\vert}
\end{align*}
is the expected number of fixed points in a uniformly at random chosen automorphism of $T$. The associated random variable $C_T$ is defined in~\eqref{eq:probCT}.
We will prove $\E(C_T)=|\CMcal{P}(T)|$ by induction on the size of $T$.

The most important observation is that only if the root of a subtree is a fixed point, its children can also be fixed points. Obviously, the root of the tree is always a fixed point.

For $|T| = 1$, the claim holds as $\E(C_T)=1$ and there is just one tree with a single node and a marker on it. For larger $T$ consider the construction of \polya trees. A \polya tree consists of a root $T_0$ and its children, which are a multiset of smaller trees. Thus, the set of children is of the form
\begin{align*}
	\{T_{1,1},\ldots,T_{1,k_1},T_{2,1},\ldots,T_{2,k_2},\ldots,T_{r,1},\ldots,T_{r,k_r} \},
\end{align*}
with $T_{i,j} \in \CMcal{T},$
and where trees with the same first index are isomorphic.
On the level of children, the possible behaviors of automorphisms are permutations within the same class of trees. In other words, an automorphism may interchange the trees $T_{1,1},\ldots,T_{1,k_1}$ in $k_1!$ many ways, etc. Here the main observation comes into play: only subtrees of which the root is a fixed point might also have other fixed points. Thus, the expected number of fixed points is given by the expected number of fixed points in a random permutation of $S_{k_i}$ times the expected number of fixed points in $T_{k_i}$. By linearity of expectation we get
\begin{align*}
	\E(C_T) = \E(C_{T_0})+\sum_{i = 0}^r \underbrace{\E(\text{\# fixed points in }S_{k_i})}_{=1} \E(C_{T_i}),
\end{align*}
where $\E(C_{T_i})=\E(C_{T_{i,j}})$ for all $1\le j\le k_i$ and $\E(C_{T_0})=1$ because the root is a fixed point of any automorphism. Since the expected number of fixed points for each permutation is $1$, we get on average $1$ representative for each class of trees. This is exactly the operation of labeling one tree among each equivalence class.
Finally, by induction the claim holds.
%
%
\end{proof}
This completes the proof of Theorem~\ref{T:2}.
\qed

\medskip
As an immediate consequence of Lemma~\ref{L:markingTc}, $t_{c,n}'(1)$ counts the number of \polya trees with $n$ nodes and a single labeled node (see OEIS~A$000107$, \cite{Sloane}).
This also explains the construction of non-empty sequences of trees in~\eqref{E:expTc}: Following the connection \cite[pp.~61--62]{Bergeronbook} one can draw a path from the root to each labeled node. The nodes on that path are the roots of a sequence of \polya trees.

\begin{Remark}
	Note that Lemma~\ref{L:markingTc} also implies that the total number of fixed points in all automorphisms of a tree is a multiple of the number of automorphisms.
\end{Remark}
\begin{Remark}
Lemma~\ref{L:markingTc} can also be proved by considering cycle-pointed \polya trees; see \cite[Section~3.2]{Kang} for a full description. Let $(T,c)$ be a cycle-pointed structure considered up to symmetry where $T$ is a \polya tree and $c$ is a cycle of an automorphism $\sigma\in \Aut(T)$. Then, the number of such cycle-pointed structures $(T,c)$ where $c$ has length $1$ is exactly the number $t_T'(1)$.

\end{Remark}
Let us analyze the $D$-forests in $T_n$ more carefully. We want to count the number of $D$-forests
that have size $m$ in a random \polya tree $T_n$. Therefore, we label such $D$-forests with an additional parameter $u$ in \eqref{eq:polyadeco}. From the bivariate generating function (\ref{E:bi2}) we can recover the probability $\mathbb{P}[\vert F_n(v)\vert=m]$ to generate a $D$-forest of size $m$ in the Boltzmann sampler from \cite{P:142}.

\subsection{Proof of Theorem~\ref{T:312}}
	The first result is a direct consequence of \eqref{eq:polyadeco}, where only vertices with weighted $D$-forests of size $m$ are marked. For the second result we differentiate both sides of (\ref{E:bi2}) and get
	\begin{align*}
		T^{[m]}_u(z,1) &= \frac{T(z)}{1-T(z)} \frac{d_m z^m}{D(z)}
		               = T_c(z) \frac{d_m z^m}{D(z)}.
	\end{align*}
	Then, the sought probability is given by
	\begin{align*}
		\PR\left[ \vert F_n(v)\vert = m \right] &= \frac{ [z^n] T^{[m]}_u(z,1)}{ [z^n] T_c(z)} \\
		                      &=  \frac{d_m \rho^m}{D(\rho)} \left( 1 + \LandauO\left(n^{-1}\right)\right).
	\end{align*}
	For the last equality we used the fact that $D(z)$ is analytic in a neighborhood of $z=\rho$.

Let $P_n(u)$ be the probability generating function for the size of a weighted $D$-forest $F_n(v)$ attached to a vertex $v$ of $C_n$ in a random \polya tree $T_n$. From the previous theorem it follows that
\begin{align*}
	P_n(u) &= \sum_{m \geq 0} \frac{ [z^n] T^{[m]}_u(z,1)}{ [z^n] T_c(z)} u^m \\
	       &= \frac{ [z^n] T_c(z) \frac{D(zu)}{D(z)}}{ [z^n] T_c(z)} \\
	       &= \frac{D(\rho u)}{D(\rho)} \left( 1 + \LandauO\left(n^{-1}\right)\right).
\end{align*}
This is exactly \cite[Eq.~(5.2)]{P:142}.
\qed


\medskip
Summarizing, we state the asymptotic probabilities that a weighted $D$-forest $F_n(v)$ in $T_n$ has size equal to or greater than $m$.
\renewcommand{\arraystretch}{1.2}
\begin{table}[htbp]
	\begin{center}
%
		\begin{tabular}{|c||c|c|}
        \hline $m$ & $\PR[ \vert F_n(v)\vert = m] \approx$ & $\PR[ \vert F_n(v)\vert \ge m] \approx$ \\
        \hline
        \hline $0$ & $0.9197$ & $1.0000$ \\
        \hline $1$ & $0.0000$ & $0.0803$ \\
        \hline $2$ & $0.0526$ & $0.0803$ \\
        \hline $3$ & $0.0119$ & $0.0277$ \\
        \hline $4$ & $0.0105$ & $0.0161$ \\
        \hline $5$ & $0.0015$ & $0.0060$ \\
        \hline $6$ & $0.0027$ & $0.0041$ \\
        \hline $7$ & $0.0003$ & $0.0014$ \\
		\hline
	\end{tabular}
	\end{center}
	\caption{The probability that a weighted $D$-forest $F_n(v)$ has size equal to or greater than $m$ when $0\le m\le 7$.}
	\label{tab:decoratedcnodes}
\end{table}
\renewcommand{\arraystretch}{1.0}


\section{Conclusion and perspectives}\label{S:last}
In this paper we provide an alternative proof of the maximal size of $D$-forests in a random \polya tree. We interpret all weights on $D$-forests and $C$-trees in terms of automorphisms associated to a \polya tree, and we derive the limiting probability that for a random node $v$ the attached $D$-forest $F_n(v)$ is of a given size.

Our work can be extended to $\Omega$-\polya trees:
For any $\Omega\subseteq \mathbb{N}_0=\{0,1,\ldots\}$ such that $0\in\Omega$ and $\{0,1\}\ne \Omega$, {\em an $\Omega$-\polya tree} is a rooted unlabeled tree considered up to symmetry and with outdegree set $\Omega$. When $\Omega=\mathbb{N}_0$, a $\mathbb{N}_0$-\polya tree is a \polya tree.
In view of the connection between Boltzmann samplers and generating functions, it comes as no surprise that the ``colored'' Boltzmann sampler from \cite{P:142} is closely related to a bivariate generating function. But the unified framework in analyzing the (bivariate) generating functions offers stronger results on the limiting distributions of the size of the $C$-trees and the maximal size of $D$-forests.

The next step is the study of shape characteristics of $D$-forests like the expected number of
(distinct) trees. The $C$-tree is the simply generated tree within a \polya tree and therefore its
shape characteristics is well-known -- when conditioned on its size. Moreover, $D$-forests
certainly show a different behavior and, though they are fairly small, they still have significant
influence on the tree. We will address these and other questions in the full version of this work.

{\bf Acknowledgements:}
\label{sec:ack}
This work was supported by the SFB project F50-03 ``Combinatorics of Tree-Like Structures and Enriched Trees''.
We also thank the three referees for their feedback.


\end{document}